\documentclass[11pt,reqno]{article}

\usepackage{amsmath,amsthm,amsfonts,amssymb,latexsym,mathrsfs,color}

\newtheorem{theorem}{Theorem}

\newtheorem{corollary}[theorem]{Corollary}
\newtheorem{proposition}[theorem]{Proposition}
\newtheorem{example}[theorem]{Example}
\newtheorem{problem}[theorem]{Open Problem}

\newcommand{\pk}{{\rm pk\,}}
\newcommand{\lpk}{{\rm pk^{\emph{l}}\,}}
\newcommand{\des}{{\rm des\,}}
\newcommand{\ades}{{\rm ades\,}}

\newcommand{\msn}{{\mathcal S}_n}

\newcommand{\rz}{{\rm RZ}}
\newcommand{\sep}{\preceq}
\newcommand{\sn}{S_n}
\newcommand{\lrf}[1]{\lfloor #1\rfloor}
\newcommand{\lrc}[1]{\lceil #1\rceil}
\newcommand{\sgn}{{\rm sgn\,}}

\title{Derivative polynomials and permutations by numbers of interior peaks and left peaks\footnote{This work is supported by the Fundamental Research Funds for the Central Universities (N100323013).}}

\author
{Shi-Mei Ma \footnote{ {\it Email address:}
shimeima@yahoo.com.cn (S.-M. Ma)} }
\date{\footnotesize Department of Information and Computing Science,
        Northeastern University at Qinhuangdao,\\ Hebei 066004,
        China}

\begin{document}

\maketitle

\begin{abstract}
Derivative polynomials in two variables are defined by repeated differentiation of the tangent and secant functions. We establish the connections between the coefficients of these derivative polynomials and the numbers of interior and left peaks over the symmetric group. Properties of the generating functions for the numbers of interior and left peaks over the symmetric group, including recurrence relations, generating functions and real-rootedness, are studied.
\bigskip\\
{\sl Keywords:}\quad Derivative polynomials; Interior peaks; Left peaks
\end{abstract}
\section{Introduction}
Let $\msn$ denote the symmetric group of all permutations of $[n]$, where $[n]=\{1,2,\ldots,n\}$.
A permutation $\pi=\pi(1)\pi(2)\cdots\pi(n)\in\msn$
is {\it alternating} if $\pi(1)>\pi(2)<\cdots \pi(n)$. In other words, $\pi(i)<\pi({i+1})$ if $i$ is even and $\pi(i)>\pi({i+1})$ if $i$ is odd. Similarly $\pi$ is {\it reverse alternating} if $\pi(1)<\pi(2)>\cdots \pi(n)$.
Let $E_n$ denote the number of alternating permutations in $\msn$. For instance, $E_4=5$, corresponding to the permutations $2143$, $3142$, $3241$, $4132$ and $4231$. The number $E_n$ is called an {\it Euler number} because Euler considered the numbers $E_{2n+1}$.
The bijection $\pi\mapsto \pi^c$ on $\msn$ defined by $\pi^c(i)=n+1-\pi(i)$ shows that $E_n$ is also the number of reverse alternating permutations in $\msn$. Alternating permutations have rich combinatorial structure and we refer the reader to the survey paper by Stanley~\cite{Stanley} for recent progress on this subject.

In 1879, Andr\'e~\cite{Andre79} observed that
\begin{align*}
 \sum_{n=0}^\infty E_n\frac{x^n}{n!}
   &= \tan x+\sec x, \\
  &= 1+x+\frac{x^2}{2!}+2\frac{x^3}{3!}+5\frac{x^4}{4!}+16\frac{x^5}{5!}+\cdots.
\end{align*}
Since the tangent is an odd function and the secant is an even function, we have
$$\sum_{n=0}^\infty E_{2n+1}\frac{x^{2n+1}}{(2n+1)!}=\tan x \quad{\text and} \quad \sum_{n=0}^\infty E_{2n}\frac{x^{2n}}{(2n)!}=\sec x.$$
For this reason the integers $E_{2n+1}$ are sometimes called the {\it tangent numbers} and the integers $E_{2n}$ are called the {\it secant numbers}.


Let $D_x$ denote the differential operator $\frac{d}{dx}$.
In 1995, Hoffman~\cite{Hoffman95} considered two sequences of {\it derivative polynomials} defined respectively by
\begin{equation}\label{derivapoly-1}
D_x^n(\tan (x))=P_n(\tan (x))\quad {\text and}\quad D_x^n(\sec (x))=\sec (x) Q_n(\tan (x))
\end{equation}
for $n\geq 0$.
From the chain rule it follows that the polynomials $P_n(u)$ satisfy $P_0(u)=u$ and $P_{n+1}(u)=(1+u^2)P_n'(u)$, and similarly $Q_0(u)=1$ and $Q_{n+1}(u)=(1+u^2)Q_n'(u)+uQ_n(u)$. In particular, the numbers $P_n(0)$ and $Q_n(0)$ are respectively the tangent and secant numbers.
Recently, several combinatorial formulas for these derivative polynomials have been extensively investigated (see \cite{Cvijovic09,Franssens07,Hoffman99,Josuat10} for instance). For example,
let
$$\tan^k (x)=\sum_{n\geq k}T(n,k)\frac{x^n}{n!}$$
and
$$\sec (x) \tan^k (x)=\sum_{n\geq k}S(n,k)\frac{x^n}{n!}.$$
The numbers $T(n,k)$ and $S(n,k)$ are respectively called the {\it tangent numbers of order $k$} (see~\cite[p.~428]{Carlitz72}) and the {\it secant numbers of order $k$} ((see~\cite[p.~305]{Carlitz75})).
Clearly, the numbers $T(n,1)$ and $S(n,0)$ are the tangent and secant numbers, respectively.
Cvijovi\'c~\cite[Theorem 1]{Cvijovic09} obtained that
$$P_n(x)=T(n,1)+\sum_{k=1}^{n+1}\frac{1}{k}T(n+1,k)x^k$$
and
$$Q_n(x)=\sum_{k=0}^nS(n,k)x^k.$$

The organization of this paper is as follows.
In Section 2, derivative polynomials in two variables are defined by repeated differentiation of the tangent and secant functions. In Section 3, we study the connections between interior peaks polynomials and left peak polynomials. In Section 4, we obtain several explicit formulas.
In Section 5, we give both central and local limit theorems for the coefficients of certain polynomials.
\section{Triangular arrays}
Since $D_x(\tan (x))=\sec^2 (x)$ and $D_x(\sec (x))=\tan (x) \sec (x)$, it is convenient to set $y=\tan (x)$ and $z=\sec (x)$. Hence we have $D_x(y)=z^2$ and $D_x(z)=yz$. It is natural to consider
the triangular arrays $(W_{n,k})_{n\geq 1,0\leq k\leq \lrf{{(n-1)}/{2}}}$, $(W_{n,k}^{\textit{l}})_{n\geq 1,0\leq k\leq \lrf{{n}/{2}}}$ and $(R_{n,k})_{n\geq 1,0\leq k\leq n}$ defined respectively by
\begin{equation}\label{derivapoly-2}
D_x^n(y)=\sum_{k=0}^{\lrf{({n-1})/{2}}}W_{n,k}y^{n-2k-1}z^{2k+2},\quad D_x^n(z)=\sum_{k=0}^{\lrf{{n}/{2}}}W_{n,k}^{\textit{l}}y^{n-2k}z^{2k+1}
\end{equation}
and
\begin{equation}
D_x^n(y+z)=\sum_{k=0}^nR_{n,k}y^{n-k}z^{k+1}.
\end{equation}

In the following discussion we will always assume that $n\geq 1$ and $0\leq k\leq n$.
By the linearity of $D_x$, we have
\begin{equation}\label{Trirelation}
R_{n,k}=\begin{cases}
W_{n,\frac{k-1}{2}} & \text{if $k$ is odd},\\
W_{n,\frac{k}{2}}^{\textit{l}} & \text{if $k$ is even}.
\end{cases}
\end{equation}

\begin{theorem}
For $n\geq 1$ and $0\leq k\leq n$, we have
\begin{equation}\label{recurrence-1}
R_{n+1,k}=(k+1)R_{n,k}+(n-k+2)R_{n,k-2},
\end{equation}
\begin{equation}\label{recurrence-11}
W_{n,k}=(2k+2)W_{n-1,k}+(n-2k)W_{n-1,k-1}
\end{equation}
and
\begin{equation}\label{recurrence-111}
W^{\textit{l}}_{n,k}=(2k+1)W^{\textit{l}}_{n-1,k}+(n-2k+1)W^{\textit{l}}_{n-1,k-1}.
\end{equation}
\end{theorem}
\begin{proof}
Note that
$$D_x^{n+1}(y+z)=D_x(D_x^n(y+z))=\sum_{k=0}^n(k+1)R_{n,k}y^{n-k+1}z^{k+1}+\sum_{k=0}^n(n-k)R_{n,k}y^{n-k-1}z^{k+3}$$
Thus we obtain~(\ref{recurrence-1}).
Similarly, we get~(\ref{recurrence-11}) and~(\ref{recurrence-111}).
\end{proof}

The numbers $W_{n,k}$ and $W_{n,k}^{\textit{l}}$ arise often in combinatorics and other branches of mathematics (see~\cite[\textsf{A008303, A008971}]{Sloane}). In~\cite{Petersen06,Petersen07}, Petersen studied the peak statistic over $\msn$. Let $\pi=\pi(1)\pi(2)\cdots \pi(n)\in\msn$. An {\it interior peak} in $\pi$ is an index $i\in\{2,3,\ldots,n-1\}$ such that $\pi(i-1)<\pi(i)>\pi(i+1)$.
Let $\pk(\pi)$ denote {\it the number of
interior peaks} in $\pi$. An {\it left peak} in $\pi$ is an index $i\in[n-1]$ such that $\pi(i-1)<\pi(i)>\pi(i+1)$, where we take $\pi(0)=0$.
Let $\lpk(\pi)$ denote {\it the number of
left peaks} in $\pi$. For example, the permutation $\pi=21435$ has $\pk(\pi)=1$ and $\lpk(\pi)=2$.

Let
$$W_n(x)=\sum_{\pi\in\sn}x^{\pk(\pi)}\quad {\text and}\quad W_n^{\textit{l}}(x)=\sum_{\pi\in\sn}x^{\lpk(\pi)}.$$
Note that $\deg W_n(x)=\lrf{(n-1)/2}$ and $\deg W_n^{\textit{l}}(x)=\lrf{n/2}$.
Using (\ref{recurrence-11}) and (\ref{recurrence-111}), it is easy to verify that $W_{n,k}$ is the number of permutations in $\msn$ with $k$ interior peaks, and $W^{\textit{l}}_{n,k}$ is the number of permutations in $\msn$ with $k$ left peaks. Hence
$$W_n(x)=\sum_{k\geq0}W_{n,k}x^k\quad {\text and}\quad W_n^{\textit{l}}(x)=\sum_{k\geq0}W_{n,k}^{\textit{l}}x^k.$$
Therefore, by~(\ref{recurrence-11}), the polynomials $W_n(x)$ satisfy
\begin{equation*}
W_{n+1}(x)=(nx-x+2)W_n(x)+2x(1-x)D_xW_n(x),
\end{equation*}
with initial values $W_1(x)=1$, $W_2(x)=2$ and $W_3(x)=4+2x$.
By~(\ref{recurrence-111}), the polynomials $W_n^{\textit{l}}(x)$ satisfy
\begin{equation*}
W_{n+1}^\textit{l}(x)=(nx+1)W_n^{\textit{l}}(x)+2x(1-x)D_xW_n^{\textit{l}}(x),
\end{equation*}
with initial values $W_1^{\textit{l}}(x)=1$, $W_2^{\textit{l}}(x)=1+x$ and $W_3^{\textit{l}}(x)=1+5x$.

We now present a connection between the polynomials $W_n(x)$ and $W_n^{\textit{l}}(x)$.
\begin{theorem}\label{thm-11}
For $n\geq 1$, we have
\begin{equation*}
W_n(u)=(u-1)^{\frac{n+1}{2}}u^{-1}P_n((u-1)^{-\frac{1}{2}})
\end{equation*}
and
\begin{equation*}
W_n^{\textit{l}}(u)=(u-1)^{\frac{n}{2}}Q_n((u-1)^{-\frac{1}{2}}).
\end{equation*}
\end{theorem}
\begin{proof}
Set $y=\tan (x)$ and $z=\sec (x)$. Note that $z^2=y^2+1$.
Comparing~(\ref{derivapoly-1}) with~(\ref{derivapoly-2}), we have
\begin{equation*}
 P_n(y) = \sum_{k=0}^{\lrf{{(n-1)}/{2}}}W_{n,k}y^{n-2k-1}(1+y^2)^{k+1}
   = (y^{n-1}+y^{n+1})W_{n}(1+y^{-2}),
\end{equation*}
and
\begin{equation*}
 Q_n(y) = \sum_{k=0}^{\lrf{{n}/{2}}}W_{n,k}^{\textit{l}}y^{n-2k}(1+y^2)^k
   = y^nW_{n}^{\textit{l}}(1+y^{-2}).
\end{equation*}
Let $u=1+y^{-2}$. Then the desired results follows immediately.
\end{proof}

Let $R_n(x)=\sum_{k=0}^nR_{n,k}x^k$.
By~(\ref{recurrence-1}), we have
\begin{equation}\label{recurrence-2}
R_{n+1}(x)=(1+nx^2)R_{n}(x)+x(1-x^2)R_n'(x) \quad {\text for}\quad n\geq 1,
\end{equation}
with initial values $R_1(x)=1+x$, $R_2(x)=1+2x+x^2$ and $R_3(x)=1+4x+5x^2+2x^3$.
Set $R_{0}(x)=1$. For $0\leq n\leq 6$, the coefficients of $R_{n}(x)$
can be arranged as follows with $R_{n,k}$ in row $n$ and column $k$:
$$\begin{array}{ccccccc}
  1 &  &  &  & & &\\
  1 & 1 &  &  & & &\\
  1 & 2 & 1 &  & & &\\
  1 & 4 & 5 & 2 & &  &\\
  1 & 8 & 18 & 16 & 5 & &\\
  1 & 16 & 58 & 88 & 61& 16 &\\
  1 & 32& 179 &416 &479 &272 &61
\end{array}$$
Using~(\ref{Trirelation}), we get
\begin{equation}\label{interwin}
R_{n}(x)=xW_n(x^2)+W_n^{\textit{l}}(x^2).
\end{equation}
Therefore, $R_{n,0}=1$, $R_{n,1}=2^{n-1}$, $\sum_{k=0}^nR_{n,k}=2n!$ and $R_{n,n}=R_{n-1,n-2}=E_n$ for $n\geq 2$.

\section{Identities}
Recently, much attention has been paid to the properties of the polynomials $W_n(x)$ and $W_n^{\textit{l}}(x)$.
For a permutation $\pi\in\msn$, we define a {\it descent} to be a position $i$ such that $\pi(i)>\pi(i+1)$. Denote by $\des(\pi)$ the number of descents of $\pi$. The generating function for descents is
\begin{equation*}
A_n(x)=\sum_{\pi\in\msn}x^{\des(\pi)},
\end{equation*}
which is well known as the {\it Eulerian polynomial}.
The exponential generating function for $A_n(x)$ is
\begin{equation}\label{Axz}
A(x,z)=1+\sum_{n\geq 1}A_n(x)\frac{z^n}{n!}=\frac{(1-x)e^{z(1-x)}}{1-xe^{z(1-x)}}.
\end{equation}
By the theory of {\it enriched P-partitions}, Stembridge~\cite[Remark 4.8]{Stembridge97} showed that
\begin{equation}\label{Stembridge}
W_n\left(\frac{4x}{(1+x)^2}\right)=\frac{2^{n-1}}{(1+x)^{n-1}}A_n(x).
\end{equation}
In~\cite[Observation 3.1.2]{Petersen06}, Petersen observed that
\begin{equation}\label{Petersen}
W_n^{\textit{l}}\left(\frac{4x}{(1+x)^2}\right)=\frac{1}{(1+x)^n}\left((1-x)^n+\sum_{i=1}^n\binom{n}{i}(1-x)^{n-i}2^ixA_i(x)\right).
\end{equation}

In order to simplify the identities~(\ref{Stembridge}) and~(\ref{Petersen}),
we will present another connection between $W_n(x)$ and $W_n^{\textit{l}}(x)$.
Let $C_n$ denote the set of signed permutations of $\pm[n]$ such that $\omega(-i)=-\omega(i)$ for all $i$, where $\pm[n]=\{\pm1,\pm2,\ldots,\pm n\}$.
Let
$$C_n(x)=\sum_{\omega\in C_n}x^{\des(\omega)} \quad{\text and} \quad \widetilde{C}_n(x)=\sum_{\omega\in C_n}x^{\ades(\omega)} ,$$
where
$$\des(\omega)=|\{i\in[0,n-1]:\omega(i)>\omega({i+1})\}|$$
and
$$ \ades(\omega)=|\{i\in[0,n]:\omega(i)>\omega({i+1})\}|$$
with $\omega(0)=\omega(n+1)=0$. For example, if $$\omega=(-2,-4,6,-8,1,3,7,5),$$
then $\des(\omega)=|\{0,1,3,7\}|=4$ and $\ades(\omega)=|\{0,1,3,7,8\}|=5$.
Recently, Dilks, Petersen and Stembridge~\cite{Petersen09} studied the combinatorial expansions for the affine Eulerian polynomials $\widetilde{C}_n(x)$.
From~\cite[Cor.~5.6 and Cor.~5.7]{Petersen09}, we get
\begin{equation*}
2xW_n\left(\frac{4x}{(1+x)^2}\right)=\frac{\widetilde{C}_n(x)}{(1+x)^{n-1}}
\end{equation*}
and
\begin{equation*}
(1+x)W_n^{\textit{l}}\left(\frac{4x}{(1+x)^2}\right)=\frac{C_n(x)}{(1+x)^{n-1}}.
\end{equation*}

For $n\geq 1$, set $$T_n(x)=C_n(x^2)+\frac{1}{x}\widetilde{C}_n(x^2).$$
From their generating functions (see~\cite[Sec.~6.2]{Petersen09}), we get
$$\widetilde{C}(x,z)=1+\sum_{n\geq 1}\widetilde{C}_n(x)\frac{z^n}{n!}=\frac{1-x}{1-xe^{2z(1-x)}},$$
$$C(x,z)=1+(1+x)z+\sum_{n\geq 2}C_n(x)\frac{z^n}{n!}=\frac{(1-x)e^{z(1-x)}}{1-xe^{2z(1-x)}}.$$
Therefore,
\begin{equation}\label{Txz}
T(x,z)=1+\sum_{n\geq 1}T_n(x)\frac{z^n}{n!}=C(x^2,z)+\frac{1}{x}(\widetilde{C}(x^2,z)-1)=\frac{e^{z(1-x^2)}-x}{1-xe^{z(1-x^2)}}.
\end{equation}
Comparing~(\ref{Axz}) with~(\ref{Txz}), we get
\begin{equation}\label{TxzAxz}
x+T(x,z)=(1+x)A(x,z({1+x})).
\end{equation}
Using~(\ref{TxzAxz}), we immediately obtain the following result.
\begin{theorem}
For $n\geq 1$, we have
\begin{equation}\label{TnxAnx}
T_n(x)=(1+x)^{n+1}A_n(x).
\end{equation}
\end{theorem}

Motivated by the connections between $W_n(x)$ and $W_n^{\textit{l}}(x)$, in the following sections we will study the properties of the polynomials $R_n(x)$.
\section{Explicit formulas}
The goal of this section is to find an explicit formula for the polynomials $R_n(x)$. Let $$P(x,z)=\sum_{n\geq 0}R_n(x)\frac{z^n}{n!}.$$
Multiplying both sides of~(\ref{recurrence-2}) by $x^n/n!$ and summing over all values of $n$,
we get that $P(x,z)$ satisfies the following
partial differential equation:
\begin{equation}\label{Pxz}
x(x^2-1)\frac{\partial P(x,z)}{\partial x}+(1-x^2z)\frac{\partial P(x,z)}{\partial z}=P(x,z)+x.
\end{equation}
It is well known~\cite[\textsf{A008303, A008971}]{Sloane} that
\begin{eqnarray*}
  W(x,z) &=& \sum_{n\geq 1}W_n(x)\frac{z^n}{n!}\\
   &=& \frac{\sinh(z\sqrt{1-x})}{\sqrt{1-x}\cosh(z\sqrt{1-x})-\sinh(z\sqrt{1-x})}
\end{eqnarray*}
and
\begin{eqnarray*}
  W^{\textit{l}}(x,z) &=&1+ \sum_{n\geq 1}W_n^{\textit{l}}(x)\frac{z^n}{n!}\\
   &=& \frac{\sqrt{1-x}}{\sqrt{1-x}\cosh(z\sqrt{1-x})-\sinh(z\sqrt{1-x})}
\end{eqnarray*}
Using~(\ref{interwin}), we get
$$P(x,z)= 1+\sum_{n\geq 1}[x W_n(x^2)+W_n^{\textit{l}}(x^2)]\frac{z^n}{n!}=xW(x^2,z)+W^{\textit{l}}(x^2,z).$$
Thus
\begin{equation}\label{PxzGF}
P(x,z)=\frac{x\sinh{(z\sqrt{1-x^2})}+\sqrt{1-x^2}}{\sqrt{1-x^2}\cosh{(z\sqrt{1-x^2})}-\sinh(z\sqrt{1-x^2})},
\end{equation}
and this formula gives a solution to the partial differential equation~(\ref{Pxz}).
It should be noted that the
generating function of the polynomials $R_{n}(x)$ has also been studied by
Charalambos~\cite[p.~542]{Charalambos02}.
Recall that the hyperbolic cosine and the inverse of the hyperbolic cosine are defined respectively by
$$\cosh z=\frac{e^z+e^{-z}}{2} \quad\textrm{and}\quad  \text{arc$\cosh$}\, (x)=\ln (x+\sqrt{x^2-1})\quad {\text for}\quad x\geq 1.$$
Let $$R(x,t)=\sum_{n\geq0}R_{n+1}(x)\frac{t^{n}}{n!}.$$
Clearly, $$\frac{\partial P(x,t)}{\partial t}=R(x,t).$$
By {\it the method of characteristics} (see~\cite{Wilf} for instance), Charalambos~\cite[p.~542]{Charalambos02} obtained a two-variable generating function
\begin{equation}\label{GF}
R(x,t)=\frac{1-x^2}{x(\cosh z-1)},
\end{equation}
with $z=-t\sqrt{1-x^2}+\text{arc$\cosh$}\, (1/x)$. Comparing~(\ref{PxzGF}) with~(\ref{GF}), the latter is simpler. To obtain an explicit formula for $R_n(x)$, we will use~(\ref{GF}).
Let $\{x_i\}_{i\geq 1}$ be a sequence of variables.
The {\it partial Bell polynomials} $B_{n,k}=:B_{n,k}(x_1,x_2,\ldots,x_{n-k+1})$ are defined by the generating function
\begin{equation}\label{Bell}
\sum_{n\geq k}B_{n,k}\frac{t^n}{n!}=\frac{1}{k!}\left(\sum_{i\geq1}x_i\frac{t^i}{i!}\right)^{k},
\end{equation}
with $B_{0,0}=1$ and $B_{n,0}=0$ for $n>0$ (see~\cite[p.~133]{Comtet74}).

\begin{proposition}\label{prop-1}
Let $B_{n,k}$ be the partial Bell polynomials.
When $x_i=\left(1-x^2\right)^{\lrf{(i-1)/2}}$ for each $i\geq 1$, we have
\begin{equation}\label{TnxExp}
R_{n+1}(x)=\sum_{k=1}^n(-1)^{n-k}k!(1+x)^{k+1}B_{n,k}.
\end{equation}
\end{proposition}
\begin{proof}
Let $z=-t\sqrt{1-x^2}+\text{arc$\cosh$}\, (1/x)$. Note that
\begin{eqnarray*}
  x(\frac{e^z+e^{-z}}{2}-1) &=& \frac{e^{-t\sqrt{1-x^2}}+e^{t\sqrt{1-x^2}}}{2}+\sqrt{1-x^2}\frac{e^{-t\sqrt{1-x^2}}-e^{t\sqrt{1-x^2}}}{2}-x\\
    &=& 1-x+\sum_{i\geq 1}(-1)^iy_i\frac{t^i}{i!},
\end{eqnarray*}
where $y_i=\left(1-x^2\right)^{{\lrf{(i+1)/2}}}$.
Thus
\begin{eqnarray*}
  R(x,t) &=& \frac{1-x^2}{1-x+\sum_{i\geq 1}(-1)^iy_i\frac{t^i}{i!}} \\
         &=& \frac{1}{\frac{1}{1+x}+\sum_{i\geq 1}(-1)^ix_i\frac{t^i}{i!}},\\
         &=& \frac{1+x}{1+(1+x)\sum_{i\geq 1}(-1)^ix_i\frac{t^i}{i!}},\\
\end{eqnarray*}
where $x_i=\left(1-x^2\right)^{{\lrf{(i-1)/2}}}$.
Then the desired result follows immediately from the formula for the geometric series and~(\ref{Bell}).
\end{proof}

Here we provide an example for illustration of Proposition~\ref{prop-1}.
\begin{example}
Consider the case $n=4$, we have $$B_{4,1}=x_4,B_{4,2}=4x_1x_3+3x_2^2,B_{4,3}=6x_1^2x_2,B_{4,4}=x_1^4$$ (see~\cite[p.~307]{Comtet74} for instance).
When $x_i=\left(1-x^2\right)^{{\lrf{(i-1)/2}}}$ for each $i\geq 1$, we get
$$B_{4,1}=1-x^2,B_{4,2}=7-4x^2,B_{4,3}=6,B_{4,4}=1.$$ Thus
$$\sum_{k=1}^4(-1)^{4-k}k!(1+x)^{k+1}B_{4,k}=1+16x+58x^2+88x^3+61x^4+16x^5.$$
\end{example}

Recall that $$B_{n,k}(1,1,1,\ldots)=S(n,k),$$ where $S(n,k)$ is a {\it Stirling
number of the second kind} (see~\cite[p.~135]{Comtet74}). Hence when $x=0$, the formula~(\ref{TnxExp}) reduces to
$$1=\sum_{k=0}^n(-1)^{n-k}k!S(n,k).$$
Moreover, when $x=1$, the formula~(\ref{TnxExp}) reduces to
\begin{equation}
(n+1)!=\sum_{k=1}^n(-1)^{n-k}k!2^{k}B_{n,k}(1,1,0,0,\ldots,0).
\end{equation}

From~(\ref{TnxExp}), we observe that $R_{n}(x)$ is divisible by $(1+x)^2$ for $n\geq 2$.
It is well known that the Eulerian polynomial $A_n(x)$ has only real nonpositive simple zeros (see~\cite[Theorem 1.33]{Bona04}). In the next Section, we will show that the polynomial $R_n(x)$ also has only real zeros.
\section{Central and local limit theorems}
Let $\sgn$ denote the sign function defined on $\mathbb{R}$, i.e.,
\begin{equation*}
\sgn x=\begin{cases}
+1 & \text{if $x>0$},\\ 0  & \text{if
$x=0$},\\ -1&\text{if
$x<0$.}
\end{cases}
\end{equation*}
Let $\rz$ denote the set of real polynomials with only real zeros.
Furthermore, denote by $\rz(I)$ the set of such polynomials all
whose zeros are in the interval $I$. Suppose that $f,F\in\rz$. Let
$\{r_i\}$ and $\{s_j\}$ be all zeros of $f$ and $F$ in nonincreasing
order respectively. We say that $f$ {\it separates} $F$, denoted by
$f\sep F$, if $\deg f\le\deg F\le\deg f+1$ and
\begin{equation*}
s_1\ge r_1\ge s_2\ge r_2\ge s_3\ge r_3\ge\cdots.
\end{equation*}
It is well known that if $f\in\rz$, then $f'\in\rz$ and $f'\sep f$.

\begin{theorem}\label{mthm-1}
For $n\ge 1$, we have
$R_n(x)\in\rz[-1,0)$ and $R_n(x)\sep R_{n+1}(x)$. More
precisely, $R_n(x)$ has $\lrc{{n}/{2}}-1$ simple zeros, and the zero $x=-1$ with the multiplicity
$\lrf{{n}/{2}}+1$.
\end{theorem}
\begin{proof}
From~(\ref{recurrence-2}), we have
$\deg R_{n+1}(x) =\deg R_{n}(x)+1$ and $R_{n}(0)=1$.
Note that $R_1(x)=1+x$, $R_2(x)=(1+x)^2$, $R_3(x)=(1+x)^2(1+2x)$ and
$R_4(x)=(1+x)^3(1+5x)$.
So it suffices to consider the case $n\geq 4$. We proceed by
induction on $n$. Let $m=\lrf{{n}/{2}}+1$.
Suppose that $$R_{n}(x)=\prod_{i=1}^k(x-r_i)(x+1)^m,$$ where $0>r_1>r_2>\ldots>r_k>-1$.

Let $g(x)=\prod_{i=1}^k(x-r_i)$ and $G(x)=R_{n+1}(x)/(x+1)^m$.
Then $\deg G(x)=\deg g(x)+1$. By~(\ref{recurrence-2}), we have
\begin{equation*}
G(x)=(1+nx^2)g(x)+x(1-x^2)\sum_{j=1}^k\frac{g(x)}{x-r_j}+mx(1-x)g(x).
\end{equation*}
Note that $\sgn G(r_i)=(-1)^i$ for $1\leq i\leq k$ and $G(-1)=(1+n-2m)g(-1)$.
Hence $G(x)$ has precisely one zero in each of $k$ intervals $(r_k,r_{k-1}),\ldots,(r_2,r_1),(r_1,0)$.
Note that
\begin{equation*}
1+n-2m=n-2\lrf{n/2}-1=\begin{cases}
0 & \text{if $n$ is odd};\\ -1  & \text{if
$n$ is even}.
\end{cases}
\end{equation*}
Thus $-1$ is a simple zero of $G(x)$ if $n$ is odd. If $n$ is even, then $\sgn G(-1)=(-1)^{k+1}$ and $G(x)$ has precisely one zero in the interval $(-1,r_k)$.
Hence $R_{n+1}(x)\in\rz[-1,0)$ and $R_n(x)\sep R_{n+1}(x)$.
This completes the proof.
\end{proof}

For $n\geq 1$,
set
\begin{equation}\label{RnxGnx}
R_n(x)=(1+x)^{\lrf{{n}/{2}}+1}G_n(x).
\end{equation}
Then the polynomial $G_n(x)$ has only positive integer coefficients.
\begin{problem}
Is the similarity between~(\ref{TnxAnx}) and~(\ref{RnxGnx}) just a coincidence?
Can either equation be given a combinatorial interpretation?
\end{problem}

Let $\{a(n,k)\}_{0\leq k\leq n}$ be a sequence of positive real numbers. It has no {\it internal zeros} if there exist no indices $i<j<k$ with $a(n,i)a(n,k)\neq0$ but $a(n,j)=0$.
Let $A_n=\sum_{k=0}^na(n,k)$. We say the sequence $\{a(n,k)\}$ satisfies a central limit theorem with mean $\mu_n$ and variance $\sigma_n^2$ provided
$$\limsup_{n\rightarrow+\infty,x\in\mathbb{R}}\left|\sum_{k=0}^{\mu_n+x\sigma_n}\frac{a(n,k)}{A_n} -\frac{1}{\sqrt{2\pi}}\int_{-\infty}^xe^{-\frac{t^2}{2}}dt\right|=0.$$
The sequence satisfies a local limit theorem on $B\in\mathbb{R}$ if
$$\limsup_{n\rightarrow+\infty,x\in B}\left|\frac{\sigma_na(n,\mu_n+x\sigma_n)}{A_n} -\frac{1}{\sqrt{2\pi}}e^{-\frac{x^2}{2}}\right|=0.$$
Recall the following Bender's theorem.
\begin{theorem}\cite{Bender73}\label{bender}
Let $\{P_n\}_{n\geq1}$ be a sequence of polynomials with only real zeros. The sequence of the coefficients of $P_n$ satisfies a central limit theorem with $$\mu_n=\frac{P_n'(1)}{P_n(1)} \quad\textrm{and}\quad
\sigma_n^2=\frac{P_n'(1)}{P_n(1)}+\frac{P_n''(1)}{P_n(1)}-\left(\frac{P_n'(1)}{P_n(1)}\right)^2,$$
provided that $\lim\limits_{n\to\infty}\sigma_n^2=+\infty$.
If the sequence of coefficients of each $P_n(x)$ has no internal zeros, then the sequence of coefficients satisfies a local limit theorem.
\end{theorem}

Combining Theorem~\ref{mthm-1} and Theorem~\ref{bender}, we obtain the following result.
\begin{theorem}\label{mthm-2}
The sequence $\{R_{n,k}\}_{0\leq k\leq n}$ satisfies both a central limit theorem and a local
local limit theorem with
$\mu_n={(2n-1)}/{3}$ and $\sigma_n^2=(8n+8)/45,$ where $n\geq 4$.
\end{theorem}
\begin{proof}
By differentiating~(\ref{recurrence-2}), we obtain the recurrence
$$x_{n+1}=(4n)n!+(n-1)x_n$$
for $x_n=R_n'(1)$, and this has the solution
$x_n=(4n-2)n!/3$ for $n\geq 2
$. By Theorem~\ref{bender}, we have $\mu_n={(2n-1)}/{3}$.
Another differentiation leads to the recurrence
$$y_{n+1}=\frac{4}{3}n!(4n^2-5n+3)+(n-3)y_n$$
for $y_n=R_n''(1)$. Set $y_n=(an^2+bn+c)n!$ and solve for $a,b,c$ to get
$y_n={n!}(40n^2-84n+56)/{45}$
for $n\geq 4$.
Hence $\sigma_n^2=(8n+8)/45,$ where $n\geq 4$.
Thus $\lim\limits_{n\to\infty}\sigma_n^2=+\infty$ as desired.
\end{proof}

Let $P(x)=\sum_{i=0}^na_ix^i$ be a polynomial. Let $m$ be an index such that $a_m=\max_{0\leq i\leq n}a_i$.
Darroch~\cite{Darroch64} showed that if $P(x)\in\rz(-\infty,0]$, then
$$\left\lfloor{\frac{P_n'(1)}{P_n(1)}}\right\rfloor\leq m\leq \left\lceil{\frac{P_n'(1)}{P_n(1)}}\right\rceil.$$
So the following result is immediate.
\begin{corollary}
Let $i$ be an index such that $R_{n,i}=\max_{0\leq k\leq n}R_{n,k}$.
If ${(2n-1)}/{3}$ is an integer, then $i={(2n-1)}/{3}$;
Otherwise, $i={\lrf{(2n-1)}/{3}}$ or $i={\lrc{(2n-1)}/{3}}$.
\end{corollary}

\section*{Acknowledgements}
The author would like to thank the referee for many detailed suggestions
leading to substantial improvement of this paper.



\end{document}